\documentclass[pdflatex,sn-mathphys-num]{sn-jnl}

\usepackage[T1]{fontenc}
\usepackage[utf8]{inputenc}
\usepackage{lmodern}
\usepackage{graphicx}
\usepackage{multirow}
\usepackage{amsmath,amssymb,amsfonts}
\usepackage{amsthm}
\usepackage{mathtools}
\usepackage{mathrsfs}
\usepackage{xcolor}
\usepackage{textcomp}
\usepackage{manyfoot}
\usepackage{booktabs}
\usepackage{tikz}
\usepackage{microtype}

\hypersetup{colorlinks=true,linkcolor=blue,citecolor=blue,urlcolor=blue}

\allowdisplaybreaks

\theoremstyle{thmstyleone}
\newtheorem{theorem}{Theorem}[section]
\newtheorem{lemma}[theorem]{Lemma}
\newtheorem{proposition}[theorem]{Proposition}

\theoremstyle{thmstylethree}
\newtheorem{definition}[theorem]{Definition}

\theoremstyle{thmstyletwo}
\newtheorem{remark}[theorem]{Remark}

\DeclareMathOperator{\capacity}{cap}
\DeclareMathOperator{\dist}{dist}
\DeclareMathOperator{\diam}{diam}
\DeclareMathOperator{\Log}{Log}
\newcommand{\D}{\mathbb D}

\newcommand{\T}{\mathbb T}
\newcommand{\C}{\mathbb C}
\newcommand{\R}{\mathbb R}
\newcommand{\Z}{\mathbb Z}
\newcommand{\HH}{\mathcal H}

\begin{document}

\title[Shortest paths in polynomial lemniscate sublevel sets]{Shortest paths in polynomial lemniscate sublevel sets and a problem of Erd\H{o}s}

\author*[1]{\fnm{Venkata Siddharth} \sur{Pendyala}}\email{venkatasiddharthpendyala@gmail.com}

\abstract{Let
\[
        f(z)=\prod_{j=1}^{n}(z-a_j)
\]
be monic, with all zeros in the closed unit disk, and put
\[
        E_f=\{z\in\C: |z|\leq 1,\ |f(z)|\leq 1\}.
\]
Let $S(n)$ be the largest possible shortest length of a path in $E_f$ joining $0$ to $\partial\D$, where the maximum is taken over all such polynomials of degree $n$. We prove that, for all sufficiently large $n$,
\[
        c\sqrt{\log n}\leq S(n)\leq \pi n
\]
with an absolute constant $c>0$. This proves the qualitative unboundedness predicted by Erd\H{o}s. The proof combines an explicit geometric maze, Green-function and Faber-polynomial estimates, analytic quantization of circle measures, and a reciprocal-sweeping upper bound.}



\maketitle

\section{Introduction}

Let
\[
        f(z)=\prod_{j=1}^{n}(z-a_j),\qquad |a_j|\leq 1,
\]
be monic, and set
\[
        E_f=\{z\in\C: |z|\leq 1,\ |f(z)|\leq 1\}.
\]
The origin belongs to $E_f$, since $|f(0)|\leq1$. Erd\H{o}s asked how large the shortest length can be of a path in this sublevel set joining $0$ to the unit circle. This is Problem 4.22 in Hayman's 1974 list, where Clunie and Netanyahu are credited with the existence of such a path and Erd\H{o}s is quoted as expecting the extremal length to tend to infinity, but slowly \cite{Hayman1974}. The same question is recorded in the Hayman--Lingham collection and in the online Erd\H{o}s Problems database \cite{HaymanLingham2019,ErdosProblems1120}. The database citation is used only for identification of the problem and its recorded formulation, not as mathematical evidence of novelty.

The result proved here is
\[
        c\sqrt{\log n}\leq S(n)\leq \pi n
\]
for all sufficiently large $n$. In particular, it proves the predicted unboundedness. No claim is made that the order $\sqrt{\log n}$ is sharp.

This path problem is distinct from the Erd\H{o}s--Herzog--Piranian problem on the length of the level lemniscate
\[
        \{z\in\C: |p(z)|=1\}
\]
for a monic polynomial $p$. That problem concerns the total one-dimensional measure of an algebraic level set in the plane; see \cite{EHP1958,EremenkoHayman1999,FryntovNazarov2009,KuznetsovaTkachev2003,Tao2025}. The present paper instead concerns the intrinsic obstruction to a path inside the filled set $\{|f|\leq1\}\cap\overline\D$, with all zeros constrained to lie in $\overline\D$.

The intersection with $\overline\D$ is only a normalization of Erd\H{o}s's original path question. If a path in the full sublevel set $\{|f|\leq1\}$ starts at $0$ and reaches $|z|=1$, then its initial segment ending at the first hitting time of $\partial\D$ lies in $\overline\D$ and remains in $\{|f|\leq1\}$. Conversely, every path counted here is a path in the full sublevel set.

The proof has two independent parts. The lower bound constructs an alternating circular maze and then realizes it as a forbidden region for a polynomial lemniscate by a Faber-polynomial separator and an analytic quadrature argument on the unit circle. The upper bound reflects the zeros by replacing $f$ with
\[
        g(z)=\prod_{j=1}^{n}(1-\overline{a_j}z),
\]
uses the identity
\[
        |1-\overline{a_j}z|^2-|z-a_j|^2=(1-|z|^2)(1-|a_j|^2),
\]
and extracts a short exit path from the nodal graph of $\log|g|$ using Crofton length estimates.

We now give the precise definition and theorem.

\begin{definition}
For $n\geq 1$, let $\mathcal P_n$ be the class of monic polynomials of degree $n$ whose zeros lie in $\overline\D$, and define
\[
        S(n)=\sup_{f\in\mathcal P_n}\inf\{\ell(\gamma):\gamma\subset E_f,\ \gamma(0)=0,\ \gamma(1)\in\partial\D\},
\]
where
\[
        E_f=\{z\in\C: |z|\leq 1,\ |f(z)|\leq 1\}.
\]
The inner infimum is taken over rectifiable paths; if no such path exists, it is interpreted as $+\infty$.
\end{definition}

\begin{theorem}[Main theorem]
There is an absolute constant $c>0$ such that, for all sufficiently large $n$,
\[
        c\sqrt{\log n}\leq S(n)\leq \pi n.
\]
Consequently,
\[
        S(n)\to\infty.
\]
\end{theorem}

\section{Preliminaries and normalizations}

We write $\D=\{z:|z|<1\}$ for the open unit disk, $\overline\D=\{z:|z|\leq 1\}$ for the closed unit disk, and $\T=\R/\mathbb Z$. If $K$ is a compact set, $\dist(z,K)$ denotes the Euclidean distance from $z$ to $K$, and $\HH^1$ denotes one-dimensional Hausdorff measure. The logarithmic capacity of a non-polar compact set $K$ is denoted by $\capacity(K)$. When $\Omega=\C\setminus K$ is Greenian and contains infinity, $g_\Omega(z,\infty)$ denotes the Green function with pole at infinity, normalized by the asymptotic formula
\[
        g_\Omega(z,\infty)=\log|z|-\log\capacity(K)+o(1),\qquad z\to\infty.
\]
The few outside inputs used below are stated where they are used, with hypotheses and references. Harnack chains, Green functions, Faber polynomials, Beurling projection, Crofton integration, and nodal graph facts are therefore not invoked as unnamed background at the points where they affect the estimates.

The following elementary observation is the reason the normalization $|a_j|\leq 1$ is so useful. Since
\[
        |f(0)|=\prod_{j=1}^{n}|a_j|\leq 1,
\]
the origin is automatically in $E_f$, and every admissible path begins at a point where no special local adjustment is required. On the other hand, the zeros may lie on the unit circle, and in the lower-bound construction they will indeed do so; this causes no difficulty, because the set $E_f$ is considered inside $\overline\D$ and the endpoint of the desired path is allowed to lie on $\partial\D$.

We shall also use the following convention. A path $\gamma:0\to\partial\D$ means a continuous map $\gamma:[0,1]\to\overline\D$ such that $\gamma(0)=0$ and $\gamma(1)\in\partial\D$. Its length is denoted by $\ell(\gamma)$, and if the path is not rectifiable then its length is $+\infty$. Thus a lower bound on length is automatically true for nonrectifiable curves, and in the lower-bound proof we may restrict attention to rectifiable curves without saying so at every turn.

\section{The alternating maze}\label{sec:maze}

Throughout the lower-bound construction, $C,c,A,a$ denote positive absolute constants. They may change from line to line, but no constant depends on $N$ or $n$.

Fix $N\geq2$, and set
\[
        r_j=\frac12+\frac{j}{4(N+1)},\qquad j=1,\ldots,N.
\]
Let
\[
        \eta_0=\frac{\pi}{8},
        \qquad
        \alpha_j=
        \begin{cases}
        0,&j\text{ even},\\
        \pi,&j\text{ odd}.
        \end{cases}
\]
On the circle $|z|=r_j$, define the closed forbidden arc
\[
        K_j=
        \left\{r_je^{i\theta}:
        \dist_{\R/2\pi\Z}(\theta,\alpha_j)\geq\eta_0
        \right\}.
\]
Put
\[
        K_N=\bigcup_{j=1}^N K_j.
\]
Thus the $j$-th gate is the open arc
\[
        \left\{r_je^{i\theta}:
        \dist_{\R/2\pi\Z}(\theta,\alpha_j)<\eta_0
        \right\}.
\]
The gates alternate between the directions $0$ and $\pi$.

\begin{lemma}[Maze length]\label{lem:maze-length}
There is an absolute constant $c_0>0$ such that every path $\gamma\subset\overline{\D}$ from $0$ to $\partial\D$ with
\[
        \gamma\cap K_N=\varnothing
\]
satisfies
\[
        \ell(\gamma)\geq c_0N.
\]
\end{lemma}

\begin{proof}
If $\gamma$ is not rectifiable, the conclusion is automatic. Assume $\gamma$ is rectifiable and write it as a continuous map $[0,1]\to\overline\D$, with $\gamma(0)=0$ and $\gamma(1)\in\partial\D$.

For each $j$, let $t_j$ be the first time such that
\[
        |\gamma(t_j)|=r_j.
\]
The times exist because $|\gamma(0)|=0$ and $|\gamma(1)|=1$. Since the radii increase with $j$, the first hitting times satisfy
\[
        t_1\leq t_2\leq\cdots\leq t_N.
\]
Because $\gamma\cap K_N=\varnothing$, the point $\gamma(t_j)$ must lie in the $j$-th gate. Hence
\[
        \dist_{\R/2\pi\Z}(\arg\gamma(t_j),\alpha_j)<\eta_0.
\]
Since $\alpha_{j+1}-\alpha_j\equiv\pi\pmod{2\pi}$, we have
\[
        \dist_{\R/2\pi\Z}
        (\arg\gamma(t_{j+1}),\arg\gamma(t_j))
        \geq \pi-2\eta_0=\frac{3\pi}{4}.
\]
Also $r_j,r_{j+1}\geq1/2$. Therefore the Euclidean distance between $\gamma(t_j)$ and $\gamma(t_{j+1})$ is bounded below by an absolute constant $c_1>0$. Hence
\[
        \ell(\gamma)
        \geq
        \sum_{j=1}^{N-1}|\gamma(t_{j+1})-\gamma(t_j)|
        \geq c_1(N-1).
\]
After decreasing $c_1$, this gives the result.
\end{proof}

We now connect the arcs into a tree. Let
\[
        p_j=r_je^{i\pi/2}.
\]
Since $\pi/2$ lies outside every gate, $p_j\in K_j$. For $j=1,\ldots,N-1$, let
\[
        I_j=\{re^{i\pi/2}:r_j\leq r\leq r_{j+1}\}.
\]
Define
\[
        \widetilde K_N=K_N\cup\bigcup_{j=1}^{N-1}I_j.
\]

\begin{lemma}[Topology, capacity, and corridor]\label{lem:maze-tree-corridor}
The set $\widetilde K_N$ is a finite planar tree, and
\[
        \widehat\C\setminus\widetilde K_N
\]
is simply connected. Moreover,
\[
        c\leq \diam(\widetilde K_N)\leq C,
        \qquad
        c\leq \capacity(\widetilde K_N)\leq C.
\]
Finally, there exists a piecewise $C^1$ curve
\[
        \sigma_N\subset\C\setminus\widetilde K_N
\]
joining $0$ to $2$ such that
\[
        \operatorname{length}(\sigma_N)\leq CN,
        \qquad
        \dist(\sigma_N,\widetilde K_N)\geq \frac{c}{N}.
\]
\end{lemma}

\begin{proof}
Each $K_j$ is a simple closed subarc of the circle $|z|=r_j$, with two endpoints at the sides of the gate. The point $p_j=r_je^{i\pi/2}$ lies in the relative interior of this arc. We regard $p_j$ as an additional vertex and split $K_j$ into the two subarcs obtained by cutting at $p_j$. With this subdivision, every connector $I_j$ is an edge joining the vertex $p_j$ to the vertex $p_{j+1}$, and it meets the previously listed edges only at its endpoints. The resulting finite embedded graph has vertices consisting of the gate endpoints and the points $p_j$, and edges consisting of the two subarcs of each $K_j$ together with the connectors $I_j$.

This graph is connected. It has no cycle: removing the connectors leaves $N$ disjoint arcs, and contracting each subdivided arc $K_j$ to a vertex turns the connector structure into the ordinary path graph $1-2-\cdots-N$; a cycle in the original graph would project either to a cycle in this path graph or to a cycle contained in one of the arcs $K_j$, both impossible. Hence $\widetilde K_N$ is a finite planar tree.

We use the following elementary planar-topology fact. If $T$ is a finite embedded tree in the sphere, then a sufficiently small closed regular neighborhood of $T$ is a closed topological disk. This is proved by induction on the number of edges: a point has a disk neighborhood; attaching one new edge to a tree at a single endpoint attaches a thin rectangle to the previous disk along one boundary interval, and the result is again a disk. Applying this to $T=\widetilde K_N$, let $U$ be such a disk neighborhood. Every compact loop in $\widehat\C\setminus\widetilde K_N$ is disjoint from a smaller regular neighborhood $U'$ of $\widetilde K_N$. Since $\widehat\C\setminus\operatorname{int}U'$ is a closed disk, the loop contracts in $\widehat\C\setminus\widetilde K_N$. Therefore
\[
        \widehat\C\setminus\widetilde K_N
\]
is simply connected.

Because every point of $\widetilde K_N$ has modulus between $1/2$ and $3/4$,
\[
        \widetilde K_N\subset\{1/2\leq |z|\leq3/4\}.
\]
Thus
\[
        \diam(\widetilde K_N)\leq \frac32
\]
and
\[
        \capacity(\widetilde K_N)
        \leq\capacity(\overline{D(0,3/4)})=\frac34.
\]
On the other hand, $\widetilde K_N$ contains $K_1$, and $K_1$ contains a circular subarc whose angular opening and radius are bounded below by absolute constants. This subarc is a rotated and scaled copy, with scale bounded below, of one fixed nondegenerate circular arc. Its logarithmic capacity is therefore bounded below by an absolute positive constant. By monotonicity of capacity,
\[
        \diam(\widetilde K_N)\geq c,
        \qquad
        \capacity(\widetilde K_N)\geq c.
\]

It remains to construct the corridor. Put
\[
        \Delta=r_{j+1}-r_j=\frac{1}{4(N+1)},
        \qquad
        s_j=\frac{r_j+r_{j+1}}2.
\]
Starting at $0$, move radially to the center of the first gate $r_1e^{i\alpha_1}$. For each $j=1,\ldots,N-1$, do the following: move radially from $r_je^{i\alpha_j}$ to $s_je^{i\alpha_j}$; move along the lower semicircle $|z|=s_j$ from angle $\alpha_j$ to angle $\alpha_{j+1}$; and move radially from $s_je^{i\alpha_{j+1}}$ to $r_{j+1}e^{i\alpha_{j+1}}$. After the last gate, move radially to radius $7/8$, rotate on $|z|=7/8$ if needed, and then move radially to $2$. This gives a piecewise $C^1$ curve $\sigma_N$.

The curve is disjoint from $\widetilde K_N$. During the $j$-th gate crossing it passes through the angular center of the gate, while $K_j$ begins at angular distance $\eta_0=\pi/8$. Hence its distance from the forbidden arc on that circle is bounded below by an absolute constant times $r_j\eta_0$, hence by an absolute constant. During the transition between $r_j$ and $r_{j+1}$, the semicircle has radius $s_j$, so its radial distance from $K_j\cup K_{j+1}$ is exactly
\[
        \frac{\Delta}{2}=\frac{1}{8(N+1)}.
\]
The connectors $I_j$ lie on the upper ray $\arg z=\pi/2$, while the transition semicircles use the lower half-plane. Hence the distance from the transition semicircles to the connectors is bounded below by an absolute constant, except near the endpoints already controlled by the gate clearance. Therefore
\[
        \dist(\sigma_N,\widetilde K_N)\geq \frac{c}{N}.
\]
The length estimate is direct. Each transition from the $j$-th gate to the $(j+1)$-st gate has two radial pieces of total length $\Delta$ and one semicircular piece of length
\[
        \pi s_j\leq \frac{3\pi}{4}.
\]
There are $N-1$ transitions, and the initial and final pieces have bounded total length. Hence
\[
        \operatorname{length}(\sigma_N)\leq CN.
\]
\end{proof}

\section{Green function estimates for the tree}\label{sec:green-tree}

Let
\[
        \Omega_N=\C\setminus\widetilde K_N.
\]
Since $\widetilde K_N$ is a continuum with simply connected complement, there is a unique conformal map
\[
        \phi_N:\Omega_N\to\{|w|>1\}
\]
such that
\[
        \phi_N(\infty)=\infty,
        \qquad
        \phi_N'(\infty)>0.
\]
Write
\[
        g_N(z)=\log|\phi_N(z)|.
\]
This is the Green function of $\Omega_N$ with pole at infinity.

\begin{lemma}[Harnack lower bound]\label{lem:harnack-lower-bound}
There is an absolute constant $C>0$ such that
\[
        g_N(0)\geq e^{-CN^2}.
\]
\end{lemma}

\begin{proof}
Let $\sigma_N$ be the corridor from Lemma~\ref{lem:maze-tree-corridor}, and let
\[
        \delta_N=\frac{c}{N}
\]
be its clearance. Set
\[
        r=\frac{\delta_N}{8}.
\]
Parametrize $\sigma_N$ by arclength and choose points
\[
        x_0=0,x_1,\ldots,x_M=2
\]
in order along $\sigma_N$, with arclength spacing at most $r$. Since
\[
        \operatorname{length}(\sigma_N)\leq C_0N,
\]
we may choose them so that
\[
        M\leq1+\frac{C_0N}{r}\leq C_1N^2.
\]
For each $\nu$,
\[
        \dist(x_\nu,\widetilde K_N)\geq\delta_N=8r,
\]
so
\[
        D(x_\nu,4r)\subset\Omega_N.
\]
Also
\[
        |x_{\nu+1}-x_\nu|\leq r,
\]
so
\[
        D(x_\nu,r)\cap D(x_{\nu+1},r)\ne\varnothing.
\]
Choose
\[
        y_\nu\in D(x_\nu,r)\cap D(x_{\nu+1},r).
\]
The function $g_N$ is positive and harmonic in each disk $D(x_\nu,4r)$. By Harnack's inequality applied in $D(x_\nu,4r)$, there is an absolute constant $A_H>1$ such that
\[
        g_N(x_\nu)\geq A_H^{-1}g_N(y_\nu).
\]
Applying Harnack again in $D(x_{\nu+1},4r)$,
\[
        g_N(y_\nu)\geq A_H^{-1}g_N(x_{\nu+1}).
\]
Thus
\[
        g_N(x_\nu)\geq A_H^{-2}g_N(x_{\nu+1}).
\]
Iterating,
\[
        g_N(0)\geq A_H^{-2M}g_N(2)
        \geq e^{-C_2N^2}g_N(2).
\]
Because
\[
        \widetilde K_N\subset\overline{D(0,3/4)},
\]
domain monotonicity of Green functions gives
\[
        g_N(2)
        \geq
        g_{\C\setminus\overline{D(0,3/4)}}(2,\infty)
        =\log\frac{2}{3/4}
        =\log\frac83.
\]
Therefore
\[
        g_N(0)\geq e^{-CN^2}.
\]
\end{proof}

We also need a uniform upper bound for $g_N$ on bounded sets.

\begin{lemma}[Uniform Green upper bound]\label{lem:green-upper-bound}
There is an absolute constant $C_G$ such that
\[
        g_N(z)\leq C_G
\]
for every $N$ and every $z\in\Omega_N$ with $|z|\leq4$.
\end{lemma}

\begin{proof}
Let $\mu_N$ be the equilibrium measure of $\widetilde K_N$. The Green function has the equilibrium-potential representation
\[
        g_N(z)=\int\log|z-\xi|\,d\mu_N(\xi)-\log\capacity(\widetilde K_N).
\]
For $|z|\leq4$ and $\xi\in\widetilde K_N\subset\{|z|\leq3/4\}$,
\[
        |z-\xi|\leq4+\frac34<5.
\]
Thus
\[
        \int\log|z-\xi|\,d\mu_N(\xi)\leq\log5.
\]
By Lemma~\ref{lem:maze-tree-corridor},
\[
        \capacity(\widetilde K_N)\geq c.
\]
Therefore
\[
        g_N(z)\leq\log5-\log c=:C_G.
\]
\end{proof}

The next estimate is the only boundary regularity input used in the lower bound. We isolate the precise external theorem first.

\begin{theorem}[Beurling projection estimate, annular harmonic-measure form]\label{thm:beurling-annular}
There is a universal constant $C_{\mathrm B}>0$ with the following property. Let $0<r<1/4$, let $E\subset\overline\D$ be a continuum meeting both circles $|\zeta|=2r$ and $|\zeta|=1$, and let $W$ be the component of $\D\setminus E$ containing a point $x$ with $|x|=r$. Then
\[
        \omega(x,\partial\D,W)\leq C_{\mathrm B} r^{1/2}.
\]
Here $\omega(x,\partial\D,W)$ denotes harmonic measure of the part of $\partial W$ lying on $\partial\D$, evaluated in $W$ at $x$.
\end{theorem}

This is the Beurling--Nevanlinna projection theorem after scaling and rotation; see Ahlfors \cite[Chapter III, Section B, pp. 43--44]{Ahlfors1973} or Garnett--Marshall \cite[p. 105]{GarnettMarshall2005}. The theorem is quoted exactly in the form used below. Its constant is universal; it does not depend on the shape, length, number of components, or smoothness of any larger set containing $E$.

\begin{lemma}[Domain monotonicity for the outer-side harmonic measure]\label{lem:harmonic-measure-monotonicity}
Let $U\subset V\subset D(z_0,r_0)$ be domains, let $z\in U$, and assume that the distinguished boundary side is the same outer circle $\partial D(z_0,r_0)$. Then
\[
        \omega(z,\partial D(z_0,r_0),U)
        \leq
        \omega(z,\partial D(z_0,r_0),V).
\]
Here the harmonic measure means the Perron solution with boundary value $1$ on the part of the boundary lying on $\partial D(z_0,r_0)$ and boundary value $0$ on the remaining regular boundary.
\end{lemma}

\begin{proof}
Let $h_V(w)=\omega(w,\partial D(z_0,r_0),V)$. The function $h_V$ is harmonic in $V$, satisfies $0\leq h_V\leq1$, and has boundary limit $1$ on the outer circular side and $0$ on the remaining regular part of $\partial V$. Restricted to $U$, it is harmonic. On the outer circular side of $\partial U$ its boundary value is $1$. On every other boundary point of $U$ it has nonnegative boundary lower limit. Therefore $h_V$ is a Perron majorant for the Dirichlet problem in $U$ with boundary data equal to $1$ on the outer circle and $0$ elsewhere. The Perron solution in $U$ is the least such majorant, so it is at most $h_V(z)$. This gives the stated inequality.
\end{proof}

\begin{lemma}[Boundary square-root estimate]\label{lem:boundary-square-root}
There is an absolute constant $C_B$ with the following property. For every $N\geq2$ and every $z\in\Omega_N\cap\D$,
\[
        g_N(z)\leq C_B\,\dist(z,\widetilde K_N)^{1/2}.
\]
The constant $C_B$ is independent of the number of arcs in $\widetilde K_N$.
\end{lemma}

\begin{proof}
Let $K=\widetilde K_N$. We first record why no hidden boundary regularity is being used. The set $K$ is a finite union of analytic arcs. If $q\in K$, then in a sufficiently small disk around $q$ one of these arcs contains either a full analytic subarc through $q$ or a one-sided analytic subarc ending at $q$. In a local analytic coordinate $\zeta$ sending $q$ to $0$, this subarc contains a straight segment or a half-segment. In the complement of the segment $[-1,1]$, the function
\[
        b(\zeta)=\operatorname{Re}\sqrt{\zeta^2-1}
\]
with the branch positive at infinity is a positive harmonic barrier tending to $0$ on the segment; after scaling and, in the half-segment case, reflecting across the analytic continuation of the arc, the same square-root barrier works locally. Pulling the barrier back by the analytic coordinate gives a local barrier at $q$ for $\C\setminus K$. Thus every boundary point of $\Omega_N$ is regular, and the Green function $g_N$ extends continuously to $K$ with boundary value $0$.

Fix $z\in\Omega_N\cap\D$, and set
\[
        \delta=\dist(z,K).
\]
Choose $z_0\in K$ such that
\[
        |z-z_0|=\delta.
\]
Since $K$ has diameter bounded below by an absolute constant, there is a point of $K$ at distance at least $c_0$ from $z_0$. Fix a small absolute radius
\[
        r_0<c_0/10.
\]
If $\delta\geq r_0/4$, then the desired estimate follows from Lemma~\ref{lem:green-upper-bound} after increasing $C_B$. So assume
\[
        0<\delta<r_0/4.
\]
Because $K$ is connected and meets both $D(z_0,2\delta)$ and $\C\setminus D(z_0,r_0)$, the compact set
\[
        K\cap\{2\delta\leq |w-z_0|\leq r_0\}
\]
contains a continuum $L$ meeting both circles
\[
        \partial D(z_0,2\delta)
        \quad\text{and}\quad
        \partial D(z_0,r_0).
\]
This follows directly from connectedness: otherwise the two boundary-circle intersections would lie in different compact unions of components of the annular intersection, producing a separation of $K$.

For completeness, here is the elementary separation argument in the form used below. Let $K_-$ be the union of the components of $K\cap\overline{D(z_0,r_0)}$ that meet $\overline{D(z_0,2\delta)}$. If no component of the annular part met both boundary circles, then $K_-$ would be a compact subset of $K$ separated by a positive Euclidean distance from the nonempty compact set $K\setminus D(z_0,r_0)$. Since $K$ is compact, this would split $K$ into two disjoint relatively closed nonempty pieces, contradicting connectedness. Thus one component gives the required continuum $L$. Notice also that $z$ lies strictly inside the inner circle $|w-z_0|=2\delta$, since $|z-z_0|=\delta$.

Let $U$ be the component of
\[
        D(z_0,r_0)\setminus K
\]
which contains $z$. Let $V$ be the component of
\[
        D(z_0,r_0)\setminus L
\]
which contains $z$. Since $L\subset K$, we have
\[
        U\subset V.
\]
We now normalize the disk in order to apply Theorem~\ref{thm:beurling-annular}. Put
\[
        \zeta=\frac{w-z_0}{r_0},
        \qquad x=\frac{z-z_0}{r_0},
        \qquad E=\frac{L-z_0}{r_0}.
\]
Then $|x|=\delta/r_0<1/4$. The set $E$ is a continuum in $\overline\D$ meeting both circles
\[
        |\zeta|=2\delta/r_0
        \quad\text{and}\quad
        |\zeta|=1.
\]
Let $W$ be the component of $\D\setminus E$ containing $x$. By construction $W$ is the image of $V$ under the same affine map. The hypotheses of Theorem~\ref{thm:beurling-annular} hold with $r=\delta/r_0$, and hence
\[
        \omega(x,\partial\D,W)
        \leq C_{\mathrm B}\left(\frac{\delta}{r_0}\right)^{1/2}.
\]
Harmonic measure is conformally invariant under the affine normalization, so this is exactly
\[
        \omega(z,\partial D(z_0,r_0),V)
        \leq C_{\mathrm B}\left(\frac{\delta}{r_0}\right)^{1/2}.
\]
No feature of $K$ other than the crossing continuum $L$ is used in this step.

By Lemma~\ref{lem:harmonic-measure-monotonicity}, applied to the inclusion $U\subset V$,
\[
        \omega(z,\partial D(z_0,r_0),U)
        \leq \omega(z,\partial D(z_0,r_0),V).
\]
On the circular part $\partial D(z_0,r_0)\cap\partial U$, Lemma~\ref{lem:green-upper-bound} gives
\[
        g_N\leq C_G.
\]
On the compact part $K\cap D(z_0,r_0)$, the Green function has boundary value $0$. More explicitly, let $U_\varepsilon$ be a smooth exhaustion of $U$ obtained by deleting the $\varepsilon$-neighborhood of $K$ and staying a distance $\varepsilon$ away from the circular boundary except on a smooth approximation of the circular side. The harmonic measure representation in $U_\varepsilon$, together with $0\leq g_N\leq C_G$ on $|w-z_0|=r_0$ and $g_N\to0$ uniformly on compact pieces approaching the regular boundary portion $K$, gives
\[
        g_N(z)\leq C_G\,\omega_{U_\varepsilon}(z,\partial D(z_0,r_0)\cap\partial U_\varepsilon)+o(1).
\]
Letting $\varepsilon\downarrow0$ yields
\[
        g_N(z)\leq C_G\,\omega(z,\partial D(z_0,r_0),U).
\]
Combining the last three inequalities,
\[
        g_N(z)
        \leq C\left(\frac{\delta}{r_0}\right)^{1/2}
        \leq C_B\delta^{1/2}.
\]
This proves the lemma.
\end{proof}

\section{Green level curves and Faber polynomials}\label{sec:faber-levels}

For $R>1$, define
\[
        \Gamma_R=\phi_N^{-1}(\{|w|=R\}).
\]
Since $\phi_N^{-1}$ is conformal in $|w|>1$, the curve $\Gamma_R$ is a real-analytic Jordan curve contained in $\Omega_N$. It surrounds $\widetilde K_N$. No integral below is ever taken on $\partial\widetilde K_N$.

Write the Laurent expansion of $\phi_N$ at infinity as
\[
        \phi_N(z)=\frac{z}{\capacity(\widetilde K_N)}+\alpha_0+\frac{\alpha_1}{z}+\cdots.
\]
For $m\geq1$, define the $m$-th Faber polynomial $F_m$ to be the polynomial part of the Laurent expansion of $\phi_N(z)^m$ at infinity:
\[
        \phi_N(z)^m=F_m(z)+O(z^{-1}),\qquad z\to\infty.
\]
This definition uses only the expansion at infinity.

\begin{lemma}[Faber contour identities]\label{lem:faber-contour-identities}
Let $R>1$, and let $\Gamma_R=\phi_N^{-1}(|w|=R)$, oriented positively around $\widetilde K_N$. If $z$ lies in the bounded component of $\C\setminus\Gamma_R$, then
\[
        F_m(z)=\frac1{2\pi i}\int_{\Gamma_R}
        \frac{\phi_N(\zeta)^m}{\zeta-z}\,d\zeta.
\]
If $z$ lies in the unbounded component of $\C\setminus\Gamma_R$, then
\[
        F_m(z)-\phi_N(z)^m=
        \frac1{2\pi i}\int_{\Gamma_R}
        \frac{\phi_N(\zeta)^m}{\zeta-z}\,d\zeta.
\]
\end{lemma}

\begin{proof}
Choose $M>R$, and put
\[
        \Gamma_M=\phi_N^{-1}(|w|=M).
\]
The curves $\Gamma_R$ and $\Gamma_M$ bound an annular subdomain of $\Omega_N$. Its positively oriented boundary is
\[
        \Gamma_M-\Gamma_R.
\]
If $z$ lies between $\Gamma_R$ and $\Gamma_M$, then the residue theorem gives
\[
        \frac1{2\pi i}\int_{\Gamma_M}
        \frac{\phi_N(\zeta)^m}{\zeta-z}\,d\zeta
        -
        \frac1{2\pi i}\int_{\Gamma_R}
        \frac{\phi_N(\zeta)^m}{\zeta-z}\,d\zeta
        =\phi_N(z)^m.
\]
If $z$ lies inside $\Gamma_R$, there is no pole in the annulus, so the same difference is $0$.

Let $M\to\infty$. Since
\[
        \phi_N(\zeta)^m=F_m(\zeta)+O(\zeta^{-1}),
\]
the integral over $\Gamma_M$ tends to $F_m(z)$. Indeed, the Cauchy integral of the polynomial $F_m$ over a sufficiently large circle is $F_m(z)$, while each negative Laurent term has zero Cauchy integral in the limit. This proves both identities.
\end{proof}

Next we estimate the length of the level curve.

\begin{lemma}[Level-curve length]\label{lem:level-curve-length}
For $0<\lambda\leq1$, let
\[
        \Gamma_\lambda=\{z\in\Omega_N:g_N(z)=\lambda\}.
\]
Then
\[
        \HH^1(\Gamma_\lambda)\leq C\lambda^{-1/2}.
\]
\end{lemma}

\begin{proof}
Let
\[
        \psi_N=\phi_N^{-1}.
\]
Write its Laurent expansion in $|w|>1$ as
\[
        \psi_N(w)=aw+b_0+\frac{b_1}{w}+\frac{b_2}{w^2}+\cdots,
\]
where
\[
        a=\capacity(\widetilde K_N).
\]
By Lemma~\ref{lem:maze-tree-corridor},
\[
        0<c\leq a\leq C.
\]
The exterior area theorem for schlicht functions in the exterior disk, applied to $a^{-1}\psi_N$ \cite[Theorem 2.1]{Duren1983}, gives
\[
        \sum_{k\geq1}k|b_k|^2\leq C.
\]
The constant here is absolute: the normalization divides by $a=\capacity(\widetilde K_N)$, and Lemma~\ref{lem:maze-tree-corridor} keeps $a$ between two fixed positive constants. Thus no dependence on $N$ is hidden in this coefficient estimate.
Let
\[
        R=e^\lambda.
\]
Then
\[
        \Gamma_\lambda=\psi_N(\{|w|=R\}).
\]
Thus
\[
        \HH^1(\Gamma_\lambda)=R\int_0^{2\pi}|\psi_N'(Re^{it})|\,dt.
\]
Since
\[
        \psi_N'(w)=a-\sum_{k\geq1}kb_kw^{-k-1},
\]
orthogonality gives
\[
        \int_0^{2\pi}|\psi_N'(Re^{it})|^2\,dt
        =2\pi a^2+2\pi\sum_{k\geq1}k^2|b_k|^2R^{-2k-2}.
\]
For $R>1$,
\[
        kR^{-2k}\leq\frac{C}{R-1}.
\]
Therefore
\[
        \sum_{k\geq1}k^2|b_k|^2R^{-2k-2}
        \leq
        \frac{C}{R-1}\sum_{k\geq1}k|b_k|^2
        \leq\frac{C}{R-1}.
\]
Hence
\[
        \int_0^{2\pi}|\psi_N'(Re^{it})|^2\,dt
        \leq C\left(1+\frac1{R-1}\right).
\]
By Cauchy--Schwarz,
\[
        \HH^1(\Gamma_\lambda)
        \leq CR\left(1+\frac1{R-1}\right)^{1/2}.
\]
Since $R=e^\lambda$ and $0<\lambda\leq1$,
\[
        R-1\asymp\lambda,
\]
so
\[
        \HH^1(\Gamma_\lambda)\leq C\lambda^{-1/2}.
\]
\end{proof}

\section{The Faber separator}\label{sec:faber-separator}

Set
\[
        g_0=g_N(0).
\]
By Lemma~\ref{lem:harnack-lower-bound},
\[
        g_0\geq e^{-CN^2},
\]
and by Lemma~\ref{lem:green-upper-bound},
\[
        g_0\leq C_G.
\]
Because $D(0,1/2)\cap\widetilde K_N=\varnothing$, the function $g_N$ is positive and harmonic in $D(0,1/2)$. Harnack's inequality in that disk gives an absolute constant $c_H>0$ such that
\[
        g_N(z)\geq c_Hg_0,
        \qquad |z|\leq1/4.
\]
Choose once and for all
\[
        \theta=\frac12\min\left\{c_H,\frac{\log(4/3)}{C_G},\frac1{4C_G}\right\},
        \qquad
        \lambda=\theta g_0.
\]
Then $0<\theta<1$, $0<\lambda\leq1/4$, and $\lambda<\log(4/3)$.

\begin{lemma}[Green-level geometry]\label{lem:green-level-geometry}
Let
\[
        \Gamma=\Gamma_\lambda=\{z\in\Omega_N:g_N(z)=\lambda\}.
\]
Then
\[
        \Gamma\subset\D,
        \qquad
        \dist(\Gamma,\widetilde K_N)\geq c\lambda^2,
        \qquad
        \dist(0,\Gamma)\geq\frac14.
\]
Moreover, $0$ lies in the unbounded component of $\C\setminus\Gamma$.
\end{lemma}

\begin{proof}
Since $\widetilde K_N\subset\overline{D(0,3/4)}$, domain monotonicity gives, for $|z|=1$,
\[
        g_N(z)
        \geq g_{\C\setminus\overline{D(0,3/4)}}(z,\infty)
        =\log\frac43
        >\lambda.
\]
The curve $\Gamma=\phi_N^{-1}(|w|=e^\lambda)$ is a real-analytic Jordan curve. Its bounded complementary component is
\[
        \widetilde K_N\cup\{z\in\Omega_N:g_N(z)<\lambda\}.
\]
This set is connected and contains $\widetilde K_N\subset\D$. If it met $\C\setminus\D$, it would also meet $\partial\D$; but no point of $\partial\D$ belongs to the tree or to the sublevel set $g_N<\lambda$. Hence the bounded component lies in $\D$, and so $\Gamma\subset\D$.

For $\zeta\in\Gamma$, Lemma~\ref{lem:boundary-square-root} applies and gives
\[
        \lambda=g_N(\zeta)
        \leq C_B\dist(\zeta,\widetilde K_N)^{1/2}.
\]
Thus
\[
        \dist(\zeta,\widetilde K_N)\geq C_B^{-2}\lambda^2,
\]
which is the stated separation estimate.

Finally, if $|z|\leq1/4$, then $g_N(z)\geq c_Hg_0>\lambda$, because $\theta<c_H$. Hence $\Gamma$ is disjoint from $D(0,1/4)$. Since $g_N(0)=g_0>\lambda$, the origin is not in the bounded component described above; it is in the unbounded component.
\end{proof}

\begin{proposition}[Quantitative Faber separator]\label{prop:quantitative-separator}
Let $g_0=g_N(0)$, $\lambda=\theta g_0$, and $F_m$ be the $m$-th Faber polynomial associated to $\widetilde K_N$. There is an absolute constant $C$ such that, for every integer $m\geq1$,
\[
        |F_m|_{\widetilde K_N}\leq C\lambda^{-5/2}e^{m\lambda},
\]
\[
        |F_m(0)-\phi_N(0)^m|\leq C\lambda^{-1/2}e^{m\lambda},
\]
and
\[
        \sup_{|z|\leq2}|F_m(z)|\leq e^{Cm}.
\]
Consequently, after choosing an absolute constant $A$ sufficiently large and setting
\[
        m=\left\lceil A\frac{\log(1/\lambda)}{g_0}\right\rceil,
\]
the polynomial
\[
        G_N(z)=2\left(1-\frac{F_m(z)}{F_m(0)}\right)
        =\sum_{k=1}^{D_N}a_kz^k
\]
satisfies
\[
        G_N(0)=0,
        \qquad
        \operatorname{Re}G_N(z)>1\quad (z\in\widetilde K_N),
\]
and there are absolute constants $C_1,C_2>0$ such that
\[
        D_N\leq e^{C_1N^2},
        \qquad
        B_N:=\sum_{k=1}^{D_N}k|a_k|
        \leq \exp(e^{C_2N^2}).
\]
The constants $C,C_1,C_2$ depend only on the fixed maze parameters and on the quoted Beurling and exterior-area constants; they do not depend on $N$, $m$, or $n$.
\end{proposition}

\begin{proof}
Let $\Gamma=\Gamma_\lambda$. The proof uses only the following uniform data:
\[
        |\phi_N|=e^\lambda\quad\text{on }\Gamma,
        \qquad
        \HH^1(\Gamma)\leq C\lambda^{-1/2},
\]
from Lemma~\ref{lem:level-curve-length},
\[
        \dist(\Gamma,\widetilde K_N)\geq c\lambda^2,
        \qquad
        \dist(0,\Gamma)\geq1/4,
\]
from Lemma~\ref{lem:green-level-geometry}, and
\[
        \sup_{|z|\leq4}g_N(z)\leq C_G
\]
from Lemma~\ref{lem:green-upper-bound}. All constants in these four estimates are absolute.

First take $z\in\widetilde K_N$. Since $z$ is in the bounded component of $\C\setminus\Gamma$, Lemma~\ref{lem:faber-contour-identities} gives
\[
        F_m(z)=\frac1{2\pi i}\int_\Gamma
        \frac{\phi_N(\zeta)^m}{\zeta-z}\,d\zeta.
\]
The denominator is bounded below by the separation of $\Gamma$ from the tree. Therefore
\[
        |F_m(z)|
        \leq \frac1{2\pi}\,e^{m\lambda}\,
        \frac{\HH^1(\Gamma)}{\dist(\Gamma,\widetilde K_N)}
        \leq C\lambda^{-5/2}e^{m\lambda}.
\]
Thus
\[
        |F_m|_{\widetilde K_N}\leq C\lambda^{-5/2}e^{m\lambda}.
\]

Next, since $0$ lies in the unbounded component of $\C\setminus\Gamma$, the second contour identity gives
\[
        F_m(0)-\phi_N(0)^m
        =\frac1{2\pi i}\int_\Gamma
        \frac{\phi_N(\zeta)^m}{\zeta}\,d\zeta.
\]
Here the denominator is bounded below by $\dist(0,\Gamma)\geq1/4$, so the only geometric loss is the length of $\Gamma$:
\[
        |F_m(0)-\phi_N(0)^m|
        \leq C\lambda^{-1/2}e^{m\lambda}.
\]

Finally, for $|z|\leq2$, integrate over the fixed circle $|\zeta|=4$. This circle lies in $\Omega_N$, surrounds $\widetilde K_N$, and satisfies $|\zeta-z|\geq2$. Lemma~\ref{lem:green-upper-bound} gives $|\phi_N(\zeta)|\leq e^{C_G}$ on this circle. Hence
\[
        |F_m(z)|
        \leq C e^{C_Gm}\leq e^{Cm},
        \qquad |z|\leq2.
\]

We now choose
\[
        m=\left\lceil A\frac{\log(1/\lambda)}{g_0}\right\rceil.
\]
Since $\lambda=\theta g_0$,
\[
        m(g_0-\lambda)
        \geq A(1-\theta)\log(1/\lambda).
\]
Choose $A$ so large that, for every $0<\lambda\leq1/4$,
\[
        e^{m(g_0-\lambda)}\geq4C\lambda^{-1/2}
        \quad\text{and}\quad
        e^{m(g_0-\lambda)}\geq8C\lambda^{-5/2}.
\]
This is possible because both requirements are finite fixed power inequalities in $\lambda^{-1}$; enlarging $A$ once absorbs the constants on the whole interval $0<\lambda\leq1/4$.

Since $|\phi_N(0)|=e^{g_0}$, the preceding inequalities imply
\[
        |F_m(0)|
        \geq e^{mg_0}-C\lambda^{-1/2}e^{m\lambda}
        \geq \frac12e^{mg_0},
\]
and
\[
        \frac{|F_m|_{\widetilde K_N}}{|F_m(0)|}
        \leq
        \frac{C\lambda^{-5/2}e^{m\lambda}}{\frac12e^{mg_0}}
        \leq\frac14.
\]
Define
\[
        p_N(z)=\frac{F_m(z)}{F_m(0)},
        \qquad
        G_N(z)=2(1-p_N(z)).
\]
Then $G_N(0)=0$ and, for $z\in\widetilde K_N$,
\[
        \operatorname{Re}G_N(z)
        =2(1-\operatorname{Re}p_N(z))
        \geq2(1-|p_N(z)|)
        \geq\frac32>1.
\]

The degree of $G_N$ is at most $m$. Since $g_0\geq e^{-CN^2}$ and $\lambda=\theta g_0$,
\[
        m\leq Cg_0^{-1}\log(1/g_0)\leq e^{C_1N^2}
\]
after increasing $C_1$. Thus $D_N\leq e^{C_1N^2}$.

It remains only to bound the coefficient mass. From the fixed-circle estimate and the lower bound for $|F_m(0)|$,
\[
        \sup_{|z|\leq2}|p_N(z)|
        \leq \frac{e^{Cm}}{|F_m(0)|}
        \leq2e^{Cm},
\]
and therefore
\[
        \sup_{|z|\leq2}|G_N(z)|\leq e^{C'm}.
\]
Writing $G_N(z)=\sum_{k=1}^{D_N}a_kz^k$, Cauchy's estimate on $|z|=2$ gives
\[
        |a_k|\leq e^{C'm}2^{-k}.
\]
Consequently
\[
        B_N=\sum_{k=1}^{D_N}k|a_k|
        \leq e^{C'm}\sum_{k\geq1}k2^{-k}
        \leq e^{C''m}
        \leq \exp(e^{C_2N^2}).
\]
This proves the proposition.
\end{proof}

\begin{remark}[Constant dependencies in the lower bound]\label{rem:constant-ledger}
The constants in the lower-bound construction are chosen in the following order. The maze geometry fixes $\eta_0=\pi/8$ and gives absolute constants $c_0,c,C$ in Lemmas~\ref{lem:maze-length} and \ref{lem:maze-tree-corridor}. These constants determine the capacity lower bound and the Harnack-chain constant in Lemma~\ref{lem:harnack-lower-bound}. The capacity lower bound determines $C_G$ in Lemma~\ref{lem:green-upper-bound}. The quoted Beurling constant, together with $C_G$ and the fixed diameter/capacity bounds for the tree, determines $C_B$ in Lemma~\ref{lem:boundary-square-root}. The Harnack constant $c_H$, $C_G$, and the fixed number $\log(4/3)$ determine the single parameter
\[
        \theta=\frac12\min\left\{c_H,\frac{\log(4/3)}{C_G},\frac1{4C_G}\right\}.
\]
After $\theta$ is fixed, the exponent $A$ in the Faber degree
\[
        m=\left\lceil A\frac{\log(1/\lambda)}{g_0}\right\rceil
\]
is chosen large enough to dominate the two displayed power losses $\lambda^{-1/2}$ and $\lambda^{-5/2}$ in Proposition~\ref{prop:quantitative-separator}. This fixes $C_1$ in
\[
        D_N\leq e^{C_1N^2}.
\]
The fixed-circle Faber bound and Cauchy's estimate then fix $C_2$ in
\[
        B_N\leq \exp(e^{C_2N^2}).
\]
Finally, the analytic inverse constants $\alpha,\beta,A,a$ in Lemmas~\ref{lem:analytic-inverse} and \ref{lem:midpoint-quantization}, together with $C_1$ and $C_2$, determine $C_4$ in the condition
\[
        n\geq e^{C_4N^2}.
\]
No constant in this chain depends on $N$ or $n$.
\end{remark}

\section{Analytic quantization}\label{sec:analytic-quantization}

Set
\[
        \tau_N=\frac1{16B_N},
        \qquad
        c_k=-\tau_Nka_k,
        \qquad
        1\leq k\leq D_N.
\]
Then
\[
        \sum_{k=1}^{D_N}|c_k|
        =\tau_N\sum_{k=1}^{D_N}k|a_k|
        =\frac1{16}.
\]
Define
\[
        \rho_N(t)=1+2\operatorname{Re}\sum_{k=1}^{D_N}c_ke^{2\pi ikt}.
\]
Then $\rho_N$ has mean one and
\[
        \rho_N(t)
        \geq1-2\sum_{k=1}^{D_N}|c_k|
        =\frac78.
\]
Thus $\rho_N(t)\,dt$ is a probability measure on $\R/\mathbb Z$.

For $|z|<1$, define
\[
        U_N(z)=\int_0^1\log|z-e^{2\pi it}|\rho_N(t)\,dt.
\]
For $|z|\leq3/4$, use the principal branch of
\[
        \Log(1-ze^{-2\pi it}).
\]
Then
\[
        \log|z-e^{2\pi it}|
        =\log|1-ze^{-2\pi it}|
        =-\operatorname{Re}\sum_{m=1}^{\infty}
        \frac{z^me^{-2\pi imt}}{m}.
\]
Hence
\[
        U_N(z)
        =-\operatorname{Re}\sum_{k=1}^{D_N}\frac{c_k}{k}z^k
        =\tau_N\operatorname{Re}G_N(z).
\]
Since
\[
        \operatorname{Re}G_N(z)>1
\]
on $\widetilde K_N$, we have
\[
        U_N(z)>\tau_N,
        \qquad z\in\widetilde K_N.
\]

We now discretize $\rho_N(t)\,dt$. The point of using midpoint quantiles rather than arbitrary points is that the change of variables $s=F(t)$ converts the weighted potential into an ordinary periodic midpoint rule in the $s$-variable. The nodes are real because $F$ is an increasing bijection of the real line modulo one, so the resulting zeros remain exactly on $\partial\D$. The only quantitative parameter in the quadrature error is the Fourier degree $D_N$ of the density: the inverse map $F^{-1}$ is holomorphic in a strip of width $\beta/D_N$, and the periodic midpoint rule on such a strip has error $O(e^{-a n/D_N})$. Thus no coefficient size of $G_N$ enters this discretization except through the barrier height $\tau_N$; the dependence on $D_N$ and $\tau_N$ is kept separate below.

\begin{lemma}[Analytic inverse]\label{lem:analytic-inverse}
Let
\[
        \rho(t)=1+2\operatorname{Re}\sum_{k=1}^Dc_ke^{2\pi ikt},
        \qquad
        \sum_{k=1}^D|c_k|\leq\frac18.
\]
Let
\[
        F(t)=\int_0^t\rho(s)\,ds.
\]
Then there are absolute constants $\alpha,\beta>0$ such that:
\begin{enumerate}
\item $F$ extends to an entire function and satisfies
\[
        F(w+1)=F(w)+1.
\]
\item $F$ is injective on
\[
        \Sigma_\alpha=\{w:|\operatorname{Im}w|<\alpha/D\}.
\]
\item The image $F(\Sigma_\alpha)$ contains
\[
        \{s:|\operatorname{Im}s|<\beta/D\}.
\]
\item Therefore $F^{-1}$ is a single-valued holomorphic function on
\[
        |\operatorname{Im}s|<\beta/D.
\]
\item In this strip,
\[
        F^{-1}(s+1)=F^{-1}(s)+1
\]
and
\[
        |\operatorname{Im}F^{-1}(s)|\leq\frac{2\beta}{D}.
\]
\item For every $|z|\leq3/4$, the function
\[
        H_z(s)=\Log\left(1-ze^{-2\pi iF^{-1}(s)}\right)
\]
is holomorphic, $1$-periodic, and uniformly bounded in
\[
        |\operatorname{Im}s|<\beta/D.
\]
Moreover,
\[
        |\widehat H_z(k)|\leq Ce^{-2\pi\beta |k|/D},
        \qquad k\in\mathbb Z,
\]
uniformly for $|z|\leq3/4$.
\end{enumerate}
\end{lemma}

\begin{proof}
Put
\[
        c_{-k}=\overline{c_k},
        \qquad 1\leq k\leq D.
\]
Then
\[
        \rho(w)=1+\sum_{0<|k|\leq D}c_ke^{2\pi ikw}
\]
is entire. Define
\[
        F(w)=w+\sum_{0<|k|\leq D}
        c_k\frac{e^{2\pi ikw}-1}{2\pi ik}.
\]
Then
\[
        F'(w)=\rho(w),
        \qquad
        F(0)=0,
        \qquad
        F(w+1)=F(w)+1.
\]
Write
\[
        F(w)=w+\Phi(w).
\]
Choose $\alpha>0$ absolutely, for instance any $\alpha$ with $e^{2\pi\alpha}\leq2$. Since $\sum_{k=1}^D|c_k|\leq1/8$, this gives
\[
        2e^{2\pi\alpha}\sum_{k=1}^D|c_k|\leq\frac12.
\]
For $w\in\Sigma_\alpha$,
\[
        |\Phi'(w)|=|\rho(w)-1|
        \leq2e^{2\pi\alpha}\sum_{k=1}^D|c_k|
        \leq\frac12.
\]
If $w_1,w_2\in\Sigma_\alpha$, the segment joining them lies in $\Sigma_\alpha$, so
\[
        |\Phi(w_1)-\Phi(w_2)|\leq\frac12|w_1-w_2|.
\]
Thus
\[
        |F(w_1)-F(w_2)|\geq\frac12|w_1-w_2|.
\]
Therefore $F$ is injective on $\Sigma_\alpha$.

On the real axis,
\[
        F'(t)=\rho(t)\geq\frac78.
\]
Hence $F:\R\to\R$ is strictly increasing, and since $F(t+1)=F(t)+1$, it is a bijection from $\R$ to $\R$.

Fix
\[
        s=u+iv,
        \qquad |v|<\beta/D,
\]
where $0<\beta<\alpha/8$ will be fixed absolutely. Let $x_0\in\R$ be the unique real number with
\[
        F(x_0)=u.
\]
We solve $F(x_0+\xi)=u+iv$, or equivalently
\[
        \xi=iv-\bigl(\Phi(x_0+\xi)-\Phi(x_0)\bigr).
\]
On the disk $|\xi|\leq2|v|$, the right side maps the disk into itself because
\[
        |iv-(\Phi(x_0+\xi)-\Phi(x_0))|
        \leq |v|+\frac12|\xi|
        \leq2|v|.
\]
It is a contraction with constant at most $1/2$. Hence it has a unique fixed point. Since
\[
        |\xi|\leq2|v|<\frac{2\beta}{D}<\frac{\alpha}{D},
\]
the solution lies in $\Sigma_\alpha$. Thus
\[
        \{|\operatorname{Im}s|<\beta/D\}\subset F(\Sigma_\alpha).
\]
Because $F$ is injective on $\Sigma_\alpha$, the inverse branch is globally single-valued and holomorphic on that strip. The construction gives
\[
        |\operatorname{Im}F^{-1}(s)|\leq\frac{2\beta}{D}.
\]
The identity
\[
        F(w+1)=F(w)+1
\]
implies
\[
        F^{-1}(s+1)=F^{-1}(s)+1.
\]

Now let $|z|\leq3/4$. In the strip $|\operatorname{Im}s|<\beta/D$,
\[
        |ze^{-2\pi iF^{-1}(s)}|
        \leq\frac34 e^{2\pi|\operatorname{Im}F^{-1}(s)|}
        \leq\frac34e^{4\pi\beta}.
\]
Decrease $\beta$, if necessary, so that
\[
        \frac34e^{4\pi\beta}<1.
\]
Then the principal logarithm
\[
        H_z(s)=\Log\left(1-ze^{-2\pi iF^{-1}(s)}\right)
\]
is holomorphic and uniformly bounded in $|\operatorname{Im}s|<\beta/D$. It is $1$-periodic because $F^{-1}(s+1)=F^{-1}(s)+1$.

For $k>0$, shift the Fourier coefficient integral to the line $\operatorname{Im}s=\beta/D$; for $k<0$, shift it to the line $\operatorname{Im}s=-\beta/D$. Periodicity makes the vertical sides cancel in the usual rectangular contour. Since $H_z$ is bounded in the strip by an absolute constant, the exponential factor from $e^{-2\pi iks}$ gives
\[
        |\widehat H_z(k)|\leq Ce^{-2\pi\beta|k|/D}.
\]
The case $k=0$ is included after increasing $C$. The constants $C,\beta$ are absolute and do not depend on $D$, $n$, $N$, or $z$.
\end{proof}

\begin{lemma}[Midpoint quantization]\label{lem:midpoint-quantization}
With $\rho,F$ as in Lemma~\ref{lem:analytic-inverse}, define
\[
        t_{j,n}=F^{-1}\left(\frac{j-1/2}{n}\right),
        \qquad j=1,\ldots,n.
\]
Then there are absolute constants $A,a>0$ such that, uniformly for $|z|\leq3/4$,
\[
        \left|
        \frac1n\sum_{j=1}^n\log|z-e^{2\pi it_{j,n}}|
        -
        \int_0^1\log|z-e^{2\pi it}|\rho(t)\,dt
        \right|
        \leq Ae^{-an/D}.
\]
\end{lemma}

\begin{proof}
Let
\[
        s_j=\frac{j-1/2}{n}.
\]
The points $t_{j,n}=F^{-1}(s_j)$ are real because $F$ maps $\R$ bijectively onto $\R$. Thus the later zeros $e^{2\pi it_{j,n}}$ lie exactly on $\partial\D$.
By Lemma~\ref{lem:analytic-inverse},
\[
        H_z(s)=\sum_{k\in\mathbb Z}\widehat H_z(k)e^{2\pi iks}
\]
converges absolutely and uniformly on the real axis. Therefore
\[
        \frac1n\sum_{j=1}^nH_z(s_j)
        =
        \sum_{k\in\mathbb Z}\widehat H_z(k)
        \left(\frac1n\sum_{j=1}^ne^{2\pi ik(j-1/2)/n}\right).
\]
The finite geometric sum satisfies
\[
        \frac1n\sum_{j=1}^ne^{2\pi ik(j-1/2)/n}
        =
        \begin{cases}
        (-1)^\ell,& k=\ell n,\ \ell\in\mathbb Z,\\
        0,& n\nmid k.
        \end{cases}
\]
This is the only place where the midpoint choice is used: it makes the quadrature exact for all Fourier modes except multiples of $n$, and the surviving multiples carry the harmless factor $(-1)^\ell$. The estimate is therefore controlled solely by the Fourier tail of $H_z$ at frequencies $n,2n,3n,\ldots$.
Hence
\[
        \frac1n\sum_{j=1}^nH_z(s_j)
        =
        \sum_{\ell\in\mathbb Z}(-1)^\ell\widehat H_z(\ell n).
\]
Since
\[
        \widehat H_z(0)=\int_0^1H_z(s)\,ds,
\]
we get
\[
        \left|
        \frac1n\sum_{j=1}^nH_z(s_j)-\int_0^1H_z(s)\,ds
        \right|
        \leq
        \sum_{\ell\ne0}|\widehat H_z(\ell n)|.
\]
Using Lemma~\ref{lem:analytic-inverse},
\[
        \sum_{\ell\ne0}|\widehat H_z(\ell n)|
        \leq C\sum_{\ell\ne0}e^{-2\pi\beta|\ell|n/D}
        \leq Ae^{-an/D}.
\]
Taking real parts gives
\[
        \left|
        \frac1n\sum_{j=1}^n
        \log|z-e^{2\pi iF^{-1}(s_j)}|
        -
        \int_0^1\log|z-e^{2\pi iF^{-1}(s)}|\,ds
        \right|
        \leq Ae^{-an/D}.
\]
Finally, since $s=F(t)$ and $dF(t)=\rho(t)\,dt$,
\[
        \int_0^1\log|z-e^{2\pi iF^{-1}(s)}|\,ds
        =
        \int_0^1\log|z-e^{2\pi it}|\rho(t)\,dt.
\]
This proves the lemma.
\end{proof}

\begin{remark}[Quantization scale]
The constants in Lemma~\ref{lem:midpoint-quantization} depend only on the fixed lower bound $\rho\geq7/8$ and on the fixed analytic strip constants $\alpha,\beta$ from Lemma~\ref{lem:analytic-inverse}. They do not depend on $N$, on $n$, on the coefficients of $G_N$, or on the height $\tau_N$. The coefficient size of $G_N$ enters the construction only through
\[
        \tau_N=\frac1{16B_N}.
\]
Thus the discretization condition is exactly the requirement that the midpoint error $Ae^{-an/D_N}$ be smaller than a fixed fraction of $\tau_N$.
\end{remark}

\section{Completion of the lower bound}\label{sec:lower-completion}

\begin{proposition}[Lower bound]\label{prop:lower-bound}
There is an absolute constant $c>0$ such that, for all sufficiently large $n$,
\[
        S(n)\geq c\sqrt{\log n}.
\]
\end{proposition}

\begin{proof}
Let $G_N$ be the separator from Proposition~\ref{prop:quantitative-separator}. We have
\[
        D_N\leq e^{C_1N^2}
\]
and
\[
        B_N\leq\exp(e^{C_2N^2}).
\]
Therefore
\[
        \tau_N=\frac1{16B_N}
\]
satisfies
\[
        \log\frac1{\tau_N}\leq e^{C_3N^2}.
\]
Let
\[
        F_N(t)=\int_0^t\rho_N(s)\,ds.
\]
For $j=1,\ldots,n$, define
\[
        t_{j,n}=F_N^{-1}\left(\frac{j-1/2}{n}\right),
        \qquad
        \zeta_{j,n}=e^{2\pi it_{j,n}}.
\]
Set
\[
        f_{N,n}(z)=\prod_{j=1}^n(z-\zeta_{j,n}).
\]
Then $f_{N,n}$ is monic of degree $n$, and all its zeros lie on $\partial\D$. The discretization has therefore preserved the zero constraint in the Erd\H{o}s problem exactly; the only remaining issue is to keep the logarithmic potential error below the barrier height $\tau_N$.

By Lemma~\ref{lem:midpoint-quantization}, uniformly for $|z|\leq3/4$,
\[
        \left|
        \frac1n\log|f_{N,n}(z)|-U_N(z)
        \right|
        \leq Ae^{-an/D_N}.
\]
Choose $C_4$ so large that
\[
        n\geq e^{C_4N^2}
\]
implies
\[
        Ae^{-an/D_N}\leq\frac{\tau_N}{2}.
\]
Here is the dependence explicitly. The desired inequality is equivalent to
\[
        \frac{an}{D_N}\geq \log(2A)+\log\frac1{\tau_N}.
\]
Using
\[
        D_N\leq e^{C_1N^2}
        \qquad\text{and}\qquad
        \log\frac1{\tau_N}\leq e^{C_3N^2},
\]
it is enough to require
\[
        n\geq C e^{(C_1+C_3)N^2}.
\]
After increasing $C_4$, this follows from $n\geq e^{C_4N^2}$. Thus the discretization cost is still single-exponential in $N^2$, despite the double-exponential coefficient-mass bound, because only $\log(1/\tau_N)$ enters the comparison.
For such $n$, every $z\in\widetilde K_N$ satisfies
\[
        \frac1n\log|f_{N,n}(z)|
        >U_N(z)-\frac{\tau_N}{2}
        >\frac{\tau_N}{2}>0.
\]
Thus
\[
        \widetilde K_N\subset\{|f_{N,n}|>1\}.
\]
Therefore every path in
\[
        E_{f_{N,n}}=\{z:|z|\leq1,\ |f_{N,n}(z)|\leq1\}
\]
from $0$ to $\partial\D$ must avoid $\widetilde K_N$, hence must avoid $K_N$. By Lemma~\ref{lem:maze-length}, every such path has length at least $cN$. Therefore
\[
        S(n)\geq cN
\]
whenever
\[
        n\geq e^{C_4N^2}.
\]
Now choose
\[
        N=\left\lfloor\frac1{2\sqrt{C_4}}\sqrt{\log n}\right\rfloor.
\]
For all sufficiently large $n$, $N\geq2$ and
\[
        e^{C_4N^2}\leq n.
\]
Hence
\[
        S(n)\geq c\sqrt{\log n}.
\]
This proves the lower bound.
\end{proof}

\section{Reciprocal sweeping and the linear upper bound}\label{sec:upper-bound}

The lower bound shows how a sublevel set may be made to avoid a maze, while the upper bound shows that reciprocal sweeping always leaves a connecting path of length at most $\pi n$. The first step is the reciprocal reflection of the zeros: after this operation the allowed region is replaced by a smaller one whose logarithmic modulus is harmonic in the disk.

Let
\[
        f(z)=\prod_{j=1}^{n}(z-a_j),
        \qquad |a_j|\leq1.
\]
Define the swept polynomial by the reflected factors
\[
        \boxed{\;g(z)=\prod_{j=1}^{n}(1-\overline{a_j}z)\;} .
\]
The bar over $a_j$ is part of the construction. This is not the polynomial $\prod_j(1-a_jz)$, and the distinction is essential. For $|z|\leq1$ a direct computation gives
\[
        |1-\overline{a_j}z|^2-|z-a_j|^2=(1-|z|^2)(1-|a_j|^2)\geq0.
\]
Multiplying over $j$ yields
\[
        |f(z)|\leq |g(z)|,
        \qquad |z|\leq1,
\]
and hence
\[
        A_g:=\{z\in\C: |z|\leq1,\ |g(z)|\leq1\}\subset E_f.
\]
It is therefore enough to find a controlled path in $A_g$.

If $g\equiv1$, then the radial segment from $0$ to any boundary point lies in $A_g$ and has length one. We shall henceforth assume $g\not\equiv1$, and we write
\[
        m=\deg g\leq n.
\]
Since no zero of $g$ lies in the open disk, the function
\[
        u(z)=\log|g(z)|
\]
is a nonconstant harmonic function in $\D$, and $u(0)=0$. We do not regard $u$ as a continuous function on $\overline\D$, for $g$ may vanish at boundary points when some $|a_j|=1$. Instead we keep the closed level set in polynomial form. Define
\[
        P(x,y)=|g(x+iy)|^2-1.
\]
Then $P$ is a real polynomial of degree at most $2m$ on the plane; it may of course have smaller degree in special cases, and only this upper bound will be used. Moreover
\[
        A_g=\{z\in\C: |z|\leq1,\ P(z)\leq0\},
        \qquad
        \{z\in\D:u(z)=0\}=\{z\in\D:P(z)=0\}.
\]
Let
\[
        Z=\{z\in\C: |z|\leq1,\ P(z)=0\}.
\]

\begin{lemma}[Crofton length of the swept zero level]
With $g$ as above and $m=\deg g\geq1$,
\[
        \HH^1(Z)\leq 2\pi m.
\]
\end{lemma}

\begin{proof}
For every affine line $L$, the restriction $P|_L$ is not the zero polynomial. Indeed, if $L$ is parametrized by $z=z_0+tw$ with $|w|=1$, then, as $|t|\to\infty$,
\[
        |g(z_0+tw)|^2-1=|c_m|^2|w|^{2m}t^{2m}+O(|t|^{2m-1}),
\]
where $c_m$ is the leading coefficient of $g$. Thus no affine line is contained in $Z$, and $P|_L$ has at most $2m$ real zeros. In particular
\[
        \#(Z\cap L)\leq2m
\]
for every line $L$, apart from the harmless convention that multiple roots are counted once.

The set $Z$ is a compact real algebraic curve restricted to the closed disk. Its one-dimensional strata are countably $1$-rectifiable, and its isolated algebraic points have zero $\mathcal H^1$ measure. By the Cauchy--Crofton formula for countably $1$-rectifiable sets \cite[3.2.26]{Federer1969}, equivalently in the planar normalization of Santal\'{o} \cite[Chapter 12]{Santalo1976}, in which $L_{\theta,t}$ is the affine line with normal angle $\theta\in[0,\pi)$ and signed distance $t$,
\[
        \HH^1(Z)
        \leq \frac{1}{2}\int_0^\pi\int_{\R}\#(Z\cap L_{\theta,t})\,dt\,d\theta.
\]
Only lines with $|t|\leq1$ meet the closed unit disk. Therefore
\[
        \HH^1(Z)
        \leq \frac{1}{2}\int_0^\pi\int_{-1}^{1}2m\,dt\,d\theta
        =2\pi m.
\]
\end{proof}

We next record the nodal facts which convert this global length bound into a single exit path of half that length. Since two edge-disjoint exits lie in the same algebraic level set, at least one of them will have length no more than half the total Crofton length; this is the simple source of the factor $1/2$. The argument uses the special form of the zero set of a harmonic function in the open disk. It is not a statement about arbitrary planar graphs.

\begin{lemma}[The origin component reaches the boundary]
Let $R$ be the connected component of
\[
        Z=\{z\in\C: |z|\leq1,\ P(z)=0\}
\]
which contains $0$. Then
\[
        R\cap\partial\D\ne\varnothing.
\]
\end{lemma}

\begin{proof}
Suppose, to the contrary, that $R\cap\partial\D=\varnothing$. Since $R$ is compact, there is a positive distance between $R$ and $\partial\D$. In a neighborhood of $R$ the polynomial $g$ has no zeros, and hence $P=0$ is exactly the same set as $u=0$. Thus $R$ is a compact connected component of the interior nodal set $\{u=0\}$.

The set $Z$ is a compact real algebraic set. To see directly that it has the needed finite graph structure, replace $P$ by its square-free part; this does not change $Z$. For each irreducible factor $q$ of this square-free polynomial, the set $\{q=q_x=q_y=0\}$ is finite, since otherwise $q$ would divide both of its first partial derivatives, impossible in characteristic zero. Distinct irreducible factors have only finitely many common zeros unless they share a factor. Hence, away from finitely many algebraic singularities and finitely many intersections with $\partial\D$, the implicit function theorem makes $Z$ a real analytic one-manifold. At the finitely many exceptional points, the Weierstrass preparation theorem gives only finitely many local analytic half-branches. Thus $Z\cap\overline\D$ has finitely many components and each component is a finite union of real analytic arcs and vertices.

Since the component $R$ is disjoint from $\partial\D$, distinct compact components of this finite graph are separated by a positive distance, and there is an open set $U$ with
\[
        R\subset U\Subset\D,
        \qquad
        U\cap Z=R.
\]
On $U$ the equation $P=0$ is exactly the nodal equation $u=0$, and the preceding finite-branch decomposition writes $R$ as a finite union of real analytic open arcs together with finitely many vertices.

No vertex of this finite graph has degree one. There are no boundary vertices, because $U\Subset\D$ and $R$ is not being cut by an artificial boundary. At an interior regular zero of $u$, one analytic nodal arc passes through the point and gives two local branches; at an interior critical zero, at least four local branches meet. Thus the local nodal structure of a nonconstant harmonic function rules out a free endpoint. Nor can $R$ be a single point, since a nonzero harmonic function has no isolated zero. Hence $R$ is a finite connected graph with at least one edge and with every vertex of degree at least two. By the elementary graph-theoretic pruning argument, such a graph contains a simple cycle. Consequently $R$ contains a simple closed curve $\gamma\subset\D$.

By the Jordan curve theorem, $\gamma$ bounds a Jordan domain $H\Subset\D$. Since $\gamma\subset\{u=0\}$, the harmonic function $u$ has boundary value zero on $\partial H$. The maximum principle gives $u\equiv0$ in $H$, and unique continuation then gives $u\equiv0$ in the whole disk. This would force $|g|\equiv1$ in $\D$; since $g$ is holomorphic and $g(0)=1$, the open mapping theorem gives $g\equiv1$, contrary to the case under consideration. Hence the component of $Z$ containing $0$ must meet $\partial\D$.
\end{proof}

\begin{lemma}[Graph model for the origin component]\label{lem:origin-graph-model}
After subdividing at singular points of the real algebraic curve $P=0$, at the origin, and at the finitely many intersections with $\partial\D$, the component $R$ is a finite connected embedded graph. In this graph, every interior vertex other than possibly $0$ has degree at least $2$, the vertex $0$ has degree at least $2$, every leaf lies on $\partial\D$, and no graph cycle is contained in the open disk.
\end{lemma}

\begin{proof}
The finite real-algebraic curve decomposition used in the preceding proof applies to $Z\cap\overline\D$: after replacing $P$ by its square-free part, the implicit function theorem gives analytic arcs away from finitely many algebraic singularities, and Weierstrass preparation gives finitely many local half-branches at each exceptional point. Also, $P$ has only finitely many zeros on $\partial\D$ unless $P$ vanishes identically on the unit circle. The latter would imply $|g|=1$ on $\partial\D$; since $g$ has no zeros in $\D$ and $g(0)=1$, the maximum principle would force $g\equiv1$, contrary to the case under consideration. Thus subdivision at singular points, at $0$, and at boundary intersections gives a finite embedded graph whose edges are real analytic arcs.

At an interior regular zero of $u$, one analytic nodal arc passes through the point and gives local degree $2$. At an interior critical zero, the local expansion of a nonconstant harmonic function gives at least four nodal branches. Thus no interior vertex other than possibly $0$ has degree $1$. At the origin, if $s=\operatorname{ord}_0(g-1)\geq1$, then
\[
        u(z)=\operatorname{Re}(cz^s)+O(|z|^{s+1})
\]
with $c\ne0$, and the local nodal set has $2s\geq2$ branches. Hence $0$ has degree at least $2$. Therefore every graph leaf must lie on $\partial\D$.

Finally, a graph cycle contained in $\D$ would contain a simple closed nodal curve. By the Jordan curve theorem it bounds a domain $H\Subset\D$, and $u=0$ on $\partial H$. The maximum principle gives $u\equiv0$ in $H$, and unique continuation then gives $u\equiv0$ in $\D$, again forcing $g\equiv1$. Hence no such cycle exists.
\end{proof}

\begin{theorem}[Finite edge-Menger theorem, two-path form]\label{thm:edge-menger-two}
Let $G$ be a finite graph, let $o$ be a vertex, and let $B$ be a nonempty set of vertices not containing $o$. Then $G$ contains two edge-disjoint paths from $o$ to $B$ if and only if no single edge of $G$ separates $o$ from $B$. Equivalently, if two edge-disjoint $o$--$B$ paths do not exist, then there is an edge $e$ such that the component of $G\setminus e$ containing $o$ contains no vertex of $B$.
\end{theorem}

This is the edge version of Menger's theorem in the finite case; see Diestel \cite[Chapter 3]{Diestel2017}. Only the stated two-path specialization is used below.

\begin{lemma}[A graph exit lemma]\label{lem:graph-exit}
Let $G$ be a finite connected graph embedded in the closed disk, let $o$ be an interior vertex of degree at least $2$, and let $B$ be the set of boundary vertices. Suppose every leaf of $G$ lies in $B$ and no graph cycle is contained in the open disk. Then $G$ contains two edge-disjoint paths from $o$ to $B$.
\end{lemma}

\begin{proof}
Assume that two edge-disjoint $o$--$B$ paths do not exist. By Theorem~\ref{thm:edge-menger-two}, there is an edge $e$ such that the component $H$ of $G\setminus e$ containing $o$ contains no boundary vertex. Hence $H$ is embedded in the open disk.

Since $o$ has degree at least $2$ in $G$, the graph $H$ is not a single isolated vertex: if it were, then deleting the one edge $e$ would remove every edge incident to $o$, forcing $\deg_G(o)=1$. Thus $H$ has at least one edge.

The graph $H$ cannot contain a cycle, because any such cycle would be contained in the open disk, contrary to the hypothesis on $G$. Hence $H$ is a finite tree. Every finite tree with at least one edge has at least two leaves. At most one leaf of $H$ can fail to be a leaf of $G$, namely the endpoint of the deleted edge $e$ that lies in $H$. Therefore $G$ has at least one leaf belonging to $H$. But $H$ contains no boundary vertex, so this leaf is an interior leaf of $G$, contradicting the assumption that every leaf of $G$ lies in $B$. This contradiction proves that two edge-disjoint $o$--$B$ paths exist.
\end{proof}

\begin{lemma}[Two disjoint exits from the origin]
The component $R$ contains two edge-disjoint rectifiable arcs from $0$ to $\partial\D$. Here edge-disjoint means that, after the finite graph subdivision of Lemma~\ref{lem:origin-graph-model}, the two paths share no open graph edge. They may meet at the initial vertex $0$ and may also meet at finitely many vertices, but their one-dimensional edge interiors are disjoint.
\end{lemma}

\begin{proof}
Apply Lemma~\ref{lem:origin-graph-model} to view $R$ as a finite embedded graph and then apply Lemma~\ref{lem:graph-exit} with root $0$ and boundary set $R\cap\partial\D$. The resulting edge-disjoint graph paths are finite unions of real analytic arcs, hence are rectifiable arcs in $R$ from $0$ to $\partial\D$.
\end{proof}

\begin{proposition}[Linear upper bound]
For every $n\geq1$,
\[
        S(n)\leq \pi n.
\]
\end{proposition}

\begin{proof}
Let $f\in\mathcal P_n$ be arbitrary, and let
\[
        g(z)=\prod_{j=1}^{n}(1-\overline{a_j}z)
\]
be the swept polynomial. If $g\equiv1$, a radius gives an admissible path of length $1\leq\pi n$ in $A_g\subset E_f$. Assume therefore that $g\not\equiv1$, and put $m=\deg g$.

Let $R$ be the connected component of
\[
        Z=\{z\in\C: |z|\leq1,\ |g(z)|^2-1=0\}
\]
which contains the origin. Since $g(0)=1$, the origin belongs to $Z$, and the preceding lemmas show that $R$ contains two edge-disjoint arcs $\gamma_1,\gamma_2$ from $0$ to $\partial\D$. Since these arcs are edge-disjoint in the subdivided graph, their edge interiors are disjoint subsets of the rectifiable set $R$. A finite set of common vertices has zero $\HH^1$-measure, and the length of each graph path is the sum of the $\HH^1$-lengths of its constituent analytic edges. Therefore
\[
        \ell(\gamma_1)+\ell(\gamma_2)
        =\HH^1(\gamma_1\cup\gamma_2)
        \leq \HH^1(R)\leq \HH^1(Z)\leq2\pi m.
\]
Consequently at least one of the two arcs, call it $\gamma$, satisfies
\[
        \ell(\gamma)\leq\pi m\leq\pi n.
\]
On $Z$ one has $|g|=1$, and hence $\gamma\subset A_g$. Since $A_g\subset E_f$, the arc $\gamma$ is an admissible path for the original polynomial $f$. Because $f\in\mathcal P_n$ was arbitrary, this proves $S(n)\leq\pi n$.
\end{proof}

\begin{remark}[A branching refinement]
If $s=\operatorname{ord}_0(g-1)$, the same local nodal analysis gives $2s$ edge-disjoint exits from the origin to $\partial\D$ in the component $R$. Averaging over all these exits gives the stronger estimate
\[
        L_f\leq \frac{\pi m}{s}\leq \frac{\pi n}{s}.
\]
The proposition uses only the case $s\geq1$, but this refinement records that a higher-order first contact of $g$ with the value $1$ at the origin can only shorten the swept-level exit supplied by the argument.
\end{remark}

\section{Proof of the main theorem}

Proposition~\ref{prop:lower-bound} gives the lower bound, and the linear upper-bound proposition gives the upper bound. Hence, for all sufficiently large $n$,
\[
        c\sqrt{\log n}\leq S(n)\leq\pi n.
\]
In particular $S(n)\to\infty$.

\begin{remark}
The proof does not identify the true order of $S(n)$. It only places it between the explicit lower obstruction supplied by the maze construction and the universal linear upper bound supplied by reciprocal sweeping.
\end{remark}

\backmatter

\begin{appendices}

\section{Auxiliary normalizations}

The Harnack-chain convention used in the Green lower bound is the following: along a chain of $M$ disks with fixed relative overlap, a positive harmonic function changes by at most a factor $e^{CM}$. In the proof of Lemma~\ref{lem:harnack-lower-bound}, the disks are obtained by covering the corridor from Lemma~\ref{lem:maze-tree-corridor}.

The capacity normalization is the logarithmic one used throughout the paper. In Lemma~\ref{lem:maze-tree-corridor}, the upper capacity bound follows from monotonicity and the containment $\widetilde K_N\subset\overline{D(0,3/4)}$. The lower bound follows from monotonicity and the fact that $\widetilde K_N$ contains a rotated and scaled copy, with scale bounded below, of one fixed nondegenerate circular arc. Thus all capacities used in the Faber estimates are bounded above and below by absolute constants.

The polynomial-separation input is the Faber-polynomial construction of Sections~\ref{sec:faber-levels}--\ref{sec:faber-separator}. Beurling's projection theorem is used only in its annular harmonic-measure consequence for a continuum crossing an annulus; see Ahlfors \cite[Chapter III, Section B]{Ahlfors1973} or Garnett--Marshall \cite[p. 105]{GarnettMarshall2005}. The exterior area theorem is the coefficient estimate for schlicht functions in the exterior disk, stated for instance in Duren \cite[Theorem 2.1]{Duren1983}. The Cauchy--Crofton input in the upper bound is the rectifiable-set formula of Federer \cite[3.2.26]{Federer1969}, equivalently the planar normalization in Santal\'{o} \cite[Chapter 12]{Santalo1976}. The local nodal structure used in Section~\ref{sec:upper-bound} is the elementary Weierstrass-preparation description recorded in Appendix~B below.

\section{Nodal-set conventions}

We also spell out the elementary nodal-set convention used in Section~\ref{sec:upper-bound}. If $u$ is a nonconstant harmonic function in a neighborhood of a point $z_0$ and $u(z_0)=0$, then the first nonzero term in its Taylor expansion is the real part of a nonzero holomorphic monomial. Equivalently, after a rotation and multiplication by a nonzero real constant,
\[
        u(z)=\operatorname{Re}(z-z_0)^m+O(|z-z_0|^{m+1})
\]
for some $m\geq1$. The Weierstrass preparation theorem, or the elementary local theory of harmonic conjugates, then gives that the nodal set near $z_0$ consists of $2m$ real analytic arcs meeting at equal angles. Thus an interior nodal component has no free endpoint. In the proof of the upper bound this fact is applied only on compact subsets of $\D$, where $u=\log|g|$ is single-valued and harmonic because the swept polynomial has no zeros in the open disk.

The global finiteness used there comes from algebraicity. The set $Z=\{|g|^2-1=0\}\cap\overline\D$ is a real algebraic curve of degree at most $2n$ restricted to a compact disk. After subdividing at singularities, it is a finite union of real analytic arcs and isolated singular points. The local harmonic description rules out isolated interior points in the component through $0$, and rules out degree-one interior endpoints. Therefore, if that component were compactly contained in $\D$, it would be a finite connected graph with every vertex of degree at least two, and hence it would contain a simple cycle. The maximum-principle argument in Section~\ref{sec:upper-bound} is precisely what excludes this possibility. The graph-exit lemma packages this no-interior-cycle obstruction into the existence of two edge-disjoint exits, which is the source of the factor $1/2$ improving the Crofton estimate $2\pi n$ to the path estimate $\pi n$.

\end{appendices}


\begin{thebibliography}{20}


\bibitem{Ahlfors1973}
L. V. Ahlfors, \emph{Conformal Invariants: Topics in Geometric Function Theory}, McGraw-Hill, New York, 1973.

\bibitem{Diestel2017}
R. Diestel, \emph{Graph Theory}, fifth ed., Graduate Texts in Mathematics, vol. 173, Springer, Berlin, 2017.

\bibitem{Duren1983}
P. L. Duren, \emph{Univalent Functions}, Grundlehren der mathematischen Wissenschaften, vol. 259, Springer, New York, 1983.

\bibitem{EHP1958}
P. Erd\H{o}s, F. Herzog, and G. Piranian, Metric properties of polynomials, \emph{J. Analyse Math.} \textbf{6} (1958), 125--148.

\bibitem{ErdosProblems1120}
T. Bloom, Erd\H{o}s Problems, Problem 1120, \emph{Shortest path in a lemniscate sublevel set}, \url{https://www.erdosproblems.com/latex/1120}, accessed June 10, 2026.

\bibitem{EremenkoHayman1999}
A. Eremenko and W. Hayman, On the length of lemniscates, \emph{Michigan Math. J.} \textbf{46} (1999), no. 2, 409--415.

\bibitem{Federer1969}
H. Federer, \emph{Geometric Measure Theory}, Die Grundlehren der mathematischen Wissenschaften, vol. 153, Springer, New York, 1969.

\bibitem{FryntovNazarov2009}
A. Fryntov and F. Nazarov, New estimates for the length of the Erd\H{o}s--Herzog--Piranian lemniscate, in \emph{Linear and Complex Analysis}, Amer. Math. Soc. Transl. Ser. 2, vol. 226, American Mathematical Society, Providence, RI, 2009, pp. 49--60.

\bibitem{GarnettMarshall2005}
J. B. Garnett and D. E. Marshall, \emph{Harmonic Measure}, New Mathematical Monographs, vol. 2, Cambridge University Press, Cambridge, 2005.

\bibitem{Hayman1974}
W. K. Hayman, Research problems in function theory: new problems, in \emph{Proceedings of the Symposium on Complex Analysis (Canterbury, 1973)}, London Mathematical Society Lecture Note Series, vol. 12, Cambridge University Press, London, 1974, pp. 155--180.

\bibitem{HaymanLingham2019}
W. K. Hayman and E. F. Lingham, \emph{Research Problems in Function Theory}, Problem Books in Mathematics, Springer, Cham, 2019.

\bibitem{KuznetsovaTkachev2003}
O. S. Kuznetsova and V. G. Tkachev, Length functions of lemniscates, \emph{Manuscripta Math.} \textbf{112} (2003), no. 4, 519--538.

\bibitem{Ransford1995}
T. Ransford, \emph{Potential Theory in the Complex Plane}, London Mathematical Society Student Texts, vol. 28, Cambridge University Press, Cambridge, 1995.

\bibitem{SaffTotik1997}
E. B. Saff and V. Totik, \emph{Logarithmic Potentials with External Fields}, Grundlehren der mathematischen Wissenschaften, vol. 316, Springer, Berlin, 1997.

\bibitem{Santalo1976}
L. A. Santal\'{o}, \emph{Integral Geometry and Geometric Probability}, Encyclopedia of Mathematics and its Applications, vol. 1, Addison-Wesley, Reading, MA, 1976.

\bibitem{Tao2025}
T. Tao, The maximal length of the Erd\H{o}s--Herzog--Piranian lemniscate in high degree, arXiv:2512.12455, 2025.

\bibitem{Walsh1969}
J. L. Walsh, \emph{Interpolation and Approximation by Rational Functions in the Complex Domain}, fifth ed., American Mathematical Society Colloquium Publications, vol. 20, American Mathematical Society, Providence, RI, 1969.

\end{thebibliography}
\end{document}